\documentclass[11pt,a4paper,reqno]{amsart}
\usepackage{amsmath,amsfonts,amssymb,amsthm,mathtools,mathrsfs}
\usepackage[dvipsnames]{xcolor}
\usepackage{booktabs,array,tabularx,microtype,cite,needspace,placeins}
\usepackage[pagebackref=true,colorlinks=false]{hyperref}
\hypersetup{pdfborder={0 0 0.65},linkbordercolor={1 0 0},citebordercolor={0 0.65 0},urlbordercolor={0 0 1},pdftitle={Disproof of Wright's Problems 14 and 15 on lacunary maximal functions},pdfauthor={}}
\allowdisplaybreaks[2]
\renewcommand*{\backref}[1]{}

\renewcommand*{\backrefalt}[4]{\ifcase #1\relax\or\space{\footnotesize #2}\else\space{\footnotesize #2}\fi}

\DeclareRobustCommand{\Thmref}[1]{\mbox{\hyperref[#1]{Theorem~\ref*{#1}}}}
\DeclareRobustCommand{\Lemref}[1]{\mbox{\hyperref[#1]{Lemma~\ref*{#1}}}}
\DeclareRobustCommand{\Propref}[1]{\mbox{\hyperref[#1]{Proposition~\ref*{#1}}}}
\DeclareRobustCommand{\Corref}[1]{\mbox{\hyperref[#1]{Corollary~\ref*{#1}}}}
\DeclareRobustCommand{\Remref}[1]{\mbox{\hyperref[#1]{Remark~\ref*{#1}}}}
\DeclareRobustCommand{\Secref}[1]{\mbox{\hyperref[#1]{Section~\ref*{#1}}}}
\DeclareRobustCommand{\Subref}[1]{\mbox{\hyperref[#1]{Subsection~\ref*{#1}}}}
\DeclareRobustCommand{\Appref}[1]{\mbox{\hyperref[#1]{Appendix~\ref*{#1}}}}
\DeclareRobustCommand{\Tabref}[1]{\mbox{\hyperref[#1]{Table~\ref*{#1}}}}
\DeclareRobustCommand{\Probref}[1]{\mbox{\hyperref[#1]{Problem~\ref*{#1}}}}
\DeclareRobustCommand{\Eqref}[1]{\mbox{\hyperref[#1]{(\ref*{#1})}}}
\DeclareRobustCommand{\Extref}[2]{\mbox{\hyperlink{cite.#1}{#2}}}
\newcommand{\arxiv}[1]{\mbox{\href{https://arxiv.org/abs/#1}{arXiv\,\nolinkurl{#1}}}}
\newcommand{\doi}[1]{\mbox{\href{https://doi.org/#1}{doi:\nolinkurl{#1}}}}
\newcommand{\R}{\mathbb R}\newcommand{\Z}{\mathbb Z}
\newcommand{\dd}{\,\mathrm d}\newcommand{\ind}[1]{\mathbf1_{#1}}
\newcommand{\abs}[1]{\lvert#1\rvert}\newcommand{\norm}[1]{\lVert#1\rVert}
\DeclareMathOperator{\supp}{supp}\DeclareMathOperator{\dist}{dist}
\newcommand{\Ha}{H^1_{\mathrm{at}}}
\newcommand{\Mmu}{\mathcal M_\mu}
\newcommand{\Mmus}[1]{\mathcal M_{\mu,#1}}
\newcommand{\DM}[1]{\mathsf D_{#1}}\newcommand{\BM}[1]{\mathsf M_{#1}}
\numberwithin{equation}{section}
\newtheorem{theorem}{Theorem}[section]

\newtheorem{lemma}{Lemma}[section]
\theoremstyle{definition}\newtheorem{problem}{Problem}[section]
\theoremstyle{remark}\newtheorem{remark}{Remark}[section]
\newcolumntype{Y}{>{\raggedright\arraybackslash}X}
\begin{document}
\title[Disproof of Wright's Problems 14 and 15]{\bf Disproof of Wright's Problems 14 and 15 on lacunary maximal functions}

\author[X. Cen]{Xi Cen}
\address{Xi Cen\\
School of Science\\
China University of Mining and
Technology-Beijing\\
Beijing 100083 \\
People's Republic of China}\email{xicenmath@gmail.com}

\author[]{Guixiang Hong}
\address{Guixiang Hong\\
Institute for Advanced Study in Mathematics\\
Harbin Institute of Technology\\
Harbin
150001\\
China}
\email{gxhong@hit.edu.cn}

\author[X. Wu]{Xinfeng Wu}
\address{Xinfeng Wu\\
	School of Science\\
	China University of Mining and
	Technology-Beijing\\
	Beijing 100083 \\
	People's Republic of China}\email{wuxf@cumtb.edu.cn}

\begin{abstract}
We disprove the proposed weak endpoint bounds in James Wright's Problems 14 and 15 from the 2017 collection of harmonic-analysis problems. For every $d\ge1$, we exhibit a fixed compactly supported singular probability measure with Fourier decay of every order $0<\varepsilon<1/2$, for which the dyadic convolution maximal operator has infinite weak $L^1$ and infinite $H^1$ to weak $L^1$ operator norms. For every finite nonempty set of dyadic scales, both optimal weak constants equal the number of scales, with the Hardy constant measured in the specified atomic norm. Increasingly separated renewal digits provide uniform real-frequency Fourier decay while allowing the disjoint output regions needed for the positive and mean-zero tests. We first establish the fixed measure, then prove the two endpoint estimates using separate normalized tests. Finally we prove the two Fourier lemmas used in the construction.
\end{abstract}
\maketitle
\begin{center}\begin{minipage}{120mm}\small
\end{minipage}\end{center}

\section{Introduction}\label{sec:intro}
\subsection{Background and motivations}\label{sub:background}
Let $\mu$ be a compactly supported finite Borel measure on $\R^d$. Its dyadic dilates are defined by
\begin{equation}\label{eq:dilate}
 \int_{\R^d}\varphi(y)\dd\mu_k(y):=\int_{\R^d}\varphi(2^k y)\dd\mu(y),\qquad k\in\Z.
\end{equation}
In the notation of Wright, the lacunary convolution maximal operator associated with the fixed measure $\mu$ is
\begin{equation}\label{eq:original}
 \Mmu f(x):=\sup_{k\in\Z}\abs{f*\mu_k(x)}.
\end{equation}
For any finite nonempty $S\subseteq\Z$, we also write $\Mmus{S}f(x):=\max_{k\in S}|f*\mu_k(x)|$. The subscript $\mu$ records the underlying measure; no extra lacunary label is needed because the dyadic sequence is part of the definition.
The supremum in \Eqref{eq:original} is over the fixed dyadic sequence, not over scales chosen during the construction of $\mu$. The distinction between these two kinds of scales is essential in a counterexample.

Fourier-transform estimates provide a general method for proving maximal inequalities. The theory of Duoandikoetxea and Rubio de Francia \cite{DR} implies that a compactly supported measure with a positive power of Fourier decay has a bounded dyadic maximal operator on $L^p(\R^d)$ for $1<p\le\infty$. The weak $L^1$ endpoint asks whether this cancellation information alone controls the exceptional sets produced by many dilations. Seeger and Wright \cite{SW} discussed this question together with the associated differentiation and Hardy-space problems.

The 2017 collection of Grafakos, Oliveira e Silva, Pramanik, Seeger and Stovall \cite{GOPSS} arose from the 2016 Madison conference in honor of Michael Christ. In Section 10, contributed by James Wright, the following question appears as Problem 14 on page 15. The source formulates a question rather than a conjecture attributed jointly to the five authors. We record its quantitative meaning, restricting to positive probabilities, which already suffices for a negative answer.
\begin{problem}[Wright's Problem 14, {\cite[p.~15]{GOPSS}}]\label{prob:original}
Suppose that $\mu$ is a compactly supported finite positive Borel measure on $\R^d$ and, for some $A<\infty$ and $\varepsilon>0$,
\begin{equation}\label{eq:assumption}
 |\widehat\mu(\xi)|\le A(1+|\xi|)^{-\varepsilon},\qquad\xi\in\R^d.
\end{equation}
Does
\begin{equation}\label{eq:question}
 \|\Mmu\|_{L^1(\R^d)\to L^{1,\infty}(\R^d)}<\infty
\end{equation}
hold? Here the weak operator norm is the distributional norm defined in \Eqref{eq:weaknorm} below.
\end{problem}
\begin{problem}[Wright's Problem 15, {\cite[p.~15]{GOPSS}}]\label{prob:hardy}
Under the same Fourier-decay hypothesis, does
\begin{equation}\label{eq:hardyquestion}
 \|\Mmu\|_{H^1(\R^d)\to L^{1,\infty}(\R^d)}<\infty
\end{equation}
hold?
\end{problem}
The putative bound in \Eqref{eq:question} is allowed to depend on the entire measure. Consequently, a sequence of different measures with increasing weak constants would not by itself answer this question. The example below uses one fixed measure and a sequence of test functions.

The Hardy-space question in \Probref{prob:hardy} is logically separate from the $L^1$ question in \Probref{prob:original}. This endpoint holds for spherical surface measure by Christ's theorem \cite{Christ}. Seeger and Wright \cite[\Extref{SW}{Theorem~1.1}]{SW} prove a general Hardy-space result under a support-covering condition coupled with an optimal smoothing hypothesis. The covering and smoothing exponents in that theorem are coupled. In 1D its covering hypothesis would require
$\mathcal N_{\supp\mu}(r)\lesssim r^{-(1-s)}$ for some $0<s\le1$, where $\mathcal N_E(r)$ is the least number of radius-$r$ balls covering $E$. That condition forces $\dim_H(\supp\mu)\le1-s<1$. The 1D support constructed below has Hausdorff dimension one and therefore does not satisfy this hypothesis.

For lacunary spherical means, Seeger, Tao and Wright \cite{STW} established an $L\log\log L$ differentiation result. Cladek and Krause \cite{CK} subsequently obtained an endpoint closer to $L^1$, involving iterated logarithms. These results concern the specific geometry of the sphere. Our measures are supported on highly nonuniform Cantor sets, and the counterexample does not apply to spherical surface measure.

Fractal averaging problems supply a different comparison. \L aba and Pramanik \cite{LPmax} constructed singular measures with bounded maximal operators for suitable $p>1$ and studied differentiation by their translates and continuous dilates. A failure involving arbitrary positive dilations cannot simply be restricted to the dyadic subsequence. Uniform finite Fourier-mask approximation appears in \L aba--Pramanik \cite[\Extref{LParith}{Lemma~6.2}]{LParith} and is reproduced in Hambrook--\L aba \cite[\Extref{HL}{Lemma~7}]{HL}. The simultaneous requirement of increasingly separated integer digits is the additional constraint needed here.

Lai, Shi and Xie \cite{LSX} construct deterministic Salem spectral measures by finite-field digit sets and a homogeneous Cantor--Moran construction. Their \Extref{LSX}{Lemmas~5.2--5.3} and \Extref{LSX}{Proposition~5.4} provide the adjacent-mask framework used below, including the third factor at the upper edge of a frequency block. We reproduce the estimates with the precise hypotheses needed for our different digits. The unconstrained mask approximation and the adjacent-mask argument are prior inputs. The additional assertion established here is their compatibility with a minimum digit spacing tending to infinity, which supplies the positive and signed endpoint tests for one fixed measure.

Related 2026 work addresses continuous dilations with restricted weak bounds at exponents $p>1$ \cite{HKLL}, lacunary spherical Hardy endpoints \cite{CLLS}, and arithmetic spherical averages on $\Z^d$ \cite{LLLS}. These settings are not the fixed Euclidean measure with dyadic dilations in \Probref{prob:original} and \Probref{prob:hardy}.

\begin{table}[htbp]
\centering\small\renewcommand{\arraystretch}{1.22}
\begin{tabularx}{\textwidth}{@{}>{\raggedright\arraybackslash}p{.28\textwidth}YY@{}}\toprule
Setting & Available hypothesis or result & Distinction from the present theorem\\\midrule
General dyadic measures \cite{DR,SW,GOPSS} & Positive Fourier decay gives $L^p$, $p>1$ & \Thmref{thm:main} gives exact finite-scale weak $L^1$ failure\\
Spherical and geometric measures \cite{Christ,SW,CK,CLLS} & Hardy or Orlicz endpoints with specific geometry & Our examples also fail on genuine Hardy atoms\\
Continuous fractal averages \cite{LPmax,HKLL} & Continuous-dilation bounds and differentiation questions & Our lower bounds use the original dyadic scales\\
Salem--Moran measures \cite{LSX} & Fourier decay and spectrality & Increasing digit separation is added, not spectrality\\\bottomrule
\end{tabularx}
\caption{The hypotheses, scale families and endpoint spaces are kept separate.}\label{tab:comparison}
\end{table}

\FloatBarrier

\subsection{Main results}\label{sub:main}
For an arbitrary (possibly sublinear) operator $T$ and a normed input space $X\subseteq L^1_{\mathrm{loc}}$, fix the distributional convention
\begin{equation}\label{eq:weaknorm}
 \|T\|_{X\to L^{1,\infty}}:=\sup_{0\ne f\in X}\sup_{\lambda>0}\frac{\lambda\big|\{x:|Tf(x)|>\lambda\}\big|}{\|f\|_X}.
\end{equation}
This fixes the distributional weak quasi-norm, rather than another equivalent Lorentz normalization.

For a finite nonempty $S\subseteq\Z$, let
\begin{equation}\label{eq:weakconstant}
 \mathfrak C_S(\mu):=
 \sup_{\substack{0\ne f\in L^1(\R^d)\\\lambda>0}}
 \frac{\lambda\bigl|\{x:\Mmus{S}f(x)>\lambda\}\bigr|}{\norm f_1}.
\end{equation}
For every probability measure, the union bound gives $\mathfrak C_S(\mu)\le\#S$. The same bound applies to the atomic Hardy norm because $\norm f_1\le\norm f_{\Ha}$.

Let $\mathfrak H_S(\mu)$ denote the analogue of \Eqref{eq:weakconstant} with the denominator replaced by the atomic Hardy norm $\norm f_{\Ha}$; its normalization is specified in \Subref{sub:notation}.

\begin{theorem}[Disproof of Wright's Problems 14 and 15]\label{thm:main}
For every integer $d\ge1$, there exists a fixed nonnegative probability measure $\mu_d$ supported in $[1,5/4]^d$ and carried by a compact Lebesgue-null set, such that
\begin{equation}\label{eq:allpowers}
 \abs{\widehat{\mu_d}(\xi)}\le C_{d,\varepsilon}(1+\abs\xi)^{-\varepsilon}
 \quad\text{for every }0<\varepsilon<\tfrac12\text{ and every }\xi\in\R^d,
\end{equation}
and
\begin{equation}\label{eq:exact}
 \mathfrak C_S(\mu_d)=\mathfrak H_S(\mu_d)=\#S
 \quad\text{for every finite nonempty }S\subseteq\Z.
\end{equation}
In particular,
\[
 \|\mathcal M_{\mu_d}\|_{L^1(\R^d)\to L^{1,\infty}(\R^d)}
 =\|\mathcal M_{\mu_d}\|_{H^1(\R^d)\to L^{1,\infty}(\R^d)}=+\infty.
\]
Both \Probref{prob:original} and \Probref{prob:hardy} have negative answers. In 1D, the compact support has Hausdorff dimension and Fourier dimension equal to one.
\end{theorem}
\begin{remark}[One measure and the original quantifiers]\label{rem:quantifiers}
The measure in \Thmref{thm:main} is fixed before $\varepsilon$, $S$, $f$ and the level $\lambda$ are chosen. Only $C_{d,\varepsilon}$ depends on the exponent. Taking $\varepsilon=1/4$ already meets Wright's hypothesis in \Probref{prob:original}. The equality \Eqref{eq:exact}, with $\#S\to\infty$, rules out every finite constant $C_{\mu_d}$, even when it is allowed to depend on the entire measure. Thus the conclusion is an individual counterexample, not merely nonuniformity within a class of measures.
\end{remark}

\begin{remark}[Optimal constants and arbitrary dyadic subfamilies]\label{rem:sharpness}
The probability-convolution union bound is $\#S$ for both specified input norms, so the equalities in \Eqref{eq:exact} attain the largest possible finite-family constants. Equality refers to a supremum over test functions, not to an asserted extremizer. Replacing the atomic norm by an equivalent Hardy norm changes the constants by fixed comparison factors, but not the failure of the endpoint estimate.

For any infinite $\mathcal K\subseteq\Z$, the finite-family identity implies failure of both weak endpoints on the supremum over $\mathcal K$, including either half of the dyadic scales.
\end{remark}
\begin{remark}[Fourier smoothing and Hardy cancellation]\label{rem:hardymeaning}
The operators in \Thmref{thm:main} remain bounded on $L^p$, $1<p<\infty$, by \cite{DR,SW}. What fails is their endpoint distributional control. An $L^1$ counterexample alone gives no Hardy conclusion, since the latter input space is smaller. In \Secref{sec:mainproof}, a normalized cube atom produces signed output boxes, with all other cylinders excluded by support. This verifies the required cancellation rather than inferring it from the positive test.
\end{remark}

\subsection{Organization}\label{sub:organization}
In \Secref{sec:mainproof}, we construct the fixed measure and prove the $L^1$ and Hardy endpoint assertions of \Thmref{thm:main} in separate arguments. We state the two Fourier lemmas before using them, and give their complete proofs in \Secref{sec:auxiliary}. 

\subsection{Notation}\label{sub:notation}
We use $\widehat\tau(\xi)=\int e^{-2\pi i x\cdot\xi}\dd\tau(x)$ and $\check\psi(x)=\int e^{2\pi i x\cdot\xi}\psi(\xi)\dd\xi$. Lebesgue measure is denoted by $|E|$. The weak quasi-norm is $\norm F_{1,\infty}=\sup_{\lambda>0}\lambda|\{|F|>\lambda\}|$. The notation $U\lesssim V$ allows a fixed constant independent of stage, frequency and finite scale family, with additional dependencies displayed as subscripts. We write $U\eqsim V$ when both inequalities hold.

An $L^\infty$ Hardy atom on $\R^d$ is a function supported in an axis-parallel cube $Q$ such that $\int a=0$ and $\norm a_\infty\le|Q|^{-1}$. The space $\Ha$ consists of $L^1$-convergent sums $f=\sum_j c_j a_j$ with $\sum_j|c_j|<\infty$, with norm the infimum of that sum. This is the standard atomic realization of $H^1$ \cite{Stein}. In particular,
\begin{equation}\label{eq:atomicembed}
 \norm f_1\le\norm f_{\Ha},\qquad\norm a_{\Ha}\le1.
\end{equation}
The constants in \Eqref{eq:exact} refer to this normalization. Convolution with a finite measure is well defined on $L^1$ equivalence classes: Tonelli gives
$\int\!\int |f(x-y)|\dd\mu_k(y)\dd x=\mu(\R^d)\norm f_1$ for positive $\mu$. Applying this also to differences of representatives, and taking a countable union over $k$, makes the maximal function independent of representatives almost everywhere. All level sets below use this convention.

For an integer $N\ge2$ and a finite nonempty integer set $B$, put
\begin{equation}\label{eq:masks}
 \DM N(t):=\frac1N\sum_{j=0}^{N-1}e^{-2\pi ijt},
 \qquad \BM B(t):=\frac1{\#B}\sum_{b\in B}e^{-2\pi ibt}.
\end{equation}
Both masks are $1$-periodic, have modulus at most one, and equal one at integers. Products $P_n=N_1\cdots N_n$ always have $P_0=1$.

\section{Proof of \texorpdfstring{\Thmref{thm:main}}{Theorem 1.1}}\label{sec:mainproof}
\subsection{Construction of the fixed measure}\label{sub:fixedmeasure}
We begin with two finite Fourier lemmas. The first supplies widely separated digits whose normalized exponential sum approximates a full Dirichlet mask. The second prevents loss of decay at the boundaries between frequency blocks. We state them here and prove them in \Secref{sec:auxiliary}.
\begin{lemma}[Separated renewal digits]\label{lem:digits}
There is an absolute constant $C_0$ such that, for every integer $h\ge64$ divisible by eight and all sufficiently large integers $m$, with $N=mh$, one can find $B\subseteq\{0,\ldots,N-1\}$ satisfying
\begin{equation}\label{eq:digitlemma}
 \#B=m,\qquad\min_{b\ne b'}|b-b'|\ge h/2,
 \qquad\norm{\BM B-\DM N}_\infty\le C_0\sqrt{\frac{\log N}{m}}.
\end{equation}
The constant $C_0$ is independent of $m$ and $h$.
\end{lemma}
We prove this lemma in \Subref{sec:renewal}. For each fixed $h$, the threshold for $m$ may depend on $h$. The subsequent induction fixes $h$ before choosing $m$, which is all we need.

\begin{lemma}[Adjacent-mask estimates]\label{lem:adjacent}
Let $N_j\ge2$, let $B_j\subseteq\{0,\ldots,N_j-1\}$ be nonempty, and suppose that $\norm{\BM{B_j}-\DM{N_j}}_\infty\le\eta_j\le1$. For $n\ge2$ and $1\le t\le N_n/2$,
\begin{equation}\label{eq:pair}
 |\BM{B_{n-1}}(t)\BM{B_n}(t/N_n)|
 \lesssim\frac{N_{n-1}^{-1}+\eta_{n-1}}t+\eta_n.
\end{equation}
For $N_n/2<t<N_n$, with $b=t/N_n$,
\begin{equation}\label{eq:triple}
 |\BM{B_{n-1}}(t)\BM{B_n}(t/N_n)\BM{B_{n+1}}(b/N_{n+1})|
 \lesssim N_n^{-1}+\eta_n+\eta_{n+1}.
\end{equation}
Both constants are absolute, including at resonant frequencies and as $t\uparrow N_n$.
\end{lemma}
We use the adjacent-factor argument of \cite[\Extref{LSX}{Section~5}]{LSX}. In \Subref{sec:adjacent} we verify both estimates for our masks, including the delicate boundary $0<N_n-t<1$.

\begin{lemma}[A fixed measure with separated cylinder masses]\label{lem:fixedmeasure}
There exists a singular probability measure $\sigma$ supported in $[1,5/4]$ such that, for every $0<\varepsilon<1/2$,
\[|\widehat\sigma(\xi)|\le C_\varepsilon(1+|\xi|)^{-\varepsilon}\qquad(\xi\in\R).\]
There also exist $r_n>0$, integers $Q_n\ge1$, and $g_n\to\infty$ with $r_n\to0$ and $Q_nr_n\to0$ such that $\supp\sigma$ is covered by $Q_n$ closed intervals of length $r_n$ and equal $\sigma$-mass $Q_n^{-1}$, whose distinct left endpoints are at distance at least $g_nr_n$.
\end{lemma}
\begin{proof}[Proof of \Lemref{lem:fixedmeasure}]
\leavevmode\par\smallskip\noindent \par\medskip\noindent\textit{Step 1. The fixed digit sequence.}\quad
We construct the measure using a Cantor--Moran
scheme with stage-dependent bases $N_n$ and
digit sets $B_n$, subject to increasing digit
separation and uniform Fourier-mask approximation.
We require growing digit gaps and mask errors that decay at every fixed power below $1/2$. Set
\begin{equation}\label{eq:stageparameters}
 h_n:=8(n+10),\qquad\alpha_n:=\frac12-\frac1{n+2},\qquad P_0:=1.
\end{equation}
After $P_{n-1}$ has been chosen, choose a sufficiently
large positive integer $m_n$ and set $N_n=m_nh_n$.
Here $N_n$ is the base at stage $n$, while $m_n$
specifies the cardinality of the digit set $B_n$
to be selected using Lemma~\ref{lem:digits}.
We require
\begin{equation}\label{eq:stageconditions}
 \begin{split}
 N_n&\ge\max\{4P_{n-1},10^6\},\qquad P_n=P_{n-1}N_n,\\
 N_n^{-1}&\le P_n^{-\alpha_n},\qquad
 C_0\sqrt{\frac{h_n\log N_n}{N_n}}\le P_n^{-\alpha_n}.
 \end{split}
\end{equation}
We can impose the last inequality: with $n$, $h_n$ and $P_{n-1}$ already fixed, we have
\begin{equation}\label{eq:stagepossible}
 \frac{C_0\sqrt{h_n\log N/N}}{(P_{n-1}N)^{-\alpha_n}}
 =C_0\sqrt{h_n}\,P_{n-1}^{\alpha_n}\sqrt{\log N}\,N^{-1/(n+2)}
 \longrightarrow0,\quad (\text{as }N\to \infty).
\end{equation}
The exponent $1/(n+2)$ is strictly positive. Hence, by increasing the multiple $N_n$ of $h_n$, we meet every requirement in \Eqref{eq:stageconditions}. We also ensure $m_n$ exceeds the threshold in \Lemref{lem:digits}, and then select $B_n$ once and for all. By \Lemref{lem:digits} and \eqref{eq:stageconditions}, 
\begin{equation}\label{eq:errors}
 \eta_n:=\norm{\BM{B_n}-\DM{N_n}}_\infty\le P_n^{-\alpha_n}.
\end{equation}
In particular, the sets $B_n$ do not depend on the later choices of $\varepsilon$, the dyadic family $S$, or a test function. This order of quantifiers will let us use a \emph{single} measure for all finite families.

Next, choose independent random digits $b_n$, each uniformly distributed on the previously fixed sets $B_n$. The random series $\sum_{n\ge1}b_n/P_n$ converges uniformly over the digit choices, since
\begin{equation}\label{eq:tail}
 0\le\sum_{j>n}\frac{b_j}{P_j}
 \le\sum_{j>n}\frac{N_j-1}{P_j}
 =\sum_{j>n}\left(\frac1{P_{j-1}}-\frac1{P_j}\right)=\frac1{P_n}
\end{equation}
We denote its law by $\nu$. The telescoping identity also places its support in $[0,1]$. Independence now yields
\begin{equation}\label{eq:infiniteproduct}
 \nu=\mathop{*}_{n=1}^{\infty}\left(\frac1{m_n}\sum_{b\in B_n}\delta_{b/P_n}\right),
 \qquad \widehat\nu(\xi)=\prod_{n=1}^{\infty}\BM{B_n}(\xi/P_n).
\end{equation}
Because the partial sums converge uniformly, their characteristic functions converge at every real frequency. We will repeatedly bound the product by two or three selected factors, using that all other factors have modulus at most one.

\par\medskip\noindent\textit{Step 2. Fourier decay at all real frequencies.} We now estimate the Fourier transform without excluding any resonance. We treat the lower and upper halves of each frequency block separately.
Fix $0<\varepsilon<1/2$. Since $\alpha_j\to 1/2$ as $j\to \infty,$ we may choose $n_0\ge3$ such that $\alpha_j\ge\varepsilon$ for $j\ge n_0-1$. Then
\begin{equation}\label{eq:epsilonstage}
 \eta_j+N_j^{-1}\le2P_j^{-\varepsilon},\qquad j\ge n_0-1.
\end{equation}
For every $\xi\ge P_{n_0-1}$, choose $n\ge n_0$ with $P_{n-1}\le\xi<P_n$ and set $t=\xi/P_{n-1}\in[1,N_n)$. On the lower half, $t\le N_n/2$, by \eqref{eq:infiniteproduct} and \eqref{eq:pair},
\begin{equation}\label{eq:decaylower}
 |\widehat\nu(\xi)|\le |M_{B_{n-1}}(t)M_{B_n}(t/N_n)|\lesssim\frac{P_{n-1}^{-\varepsilon}}t+P_n^{-\varepsilon}
 \le2(P_{n-1}t)^{-\varepsilon}=2\xi^{-\varepsilon}.
\end{equation}
We used $t^{-1}\le t^{-\varepsilon}$ and $\xi<P_n$. On the upper half, the third factor $M_{B_{n+1}}(\frac{t}{N_nN_{n+1}})$
is essential near the resonant endpoint $t=N_n$. Applying the second estimate of \Lemref{lem:adjacent}, \eqref{eq:stageconditions},
\eqref{eq:errors} and $\alpha_j\ge \varepsilon$ for $j\ge n_0$, we obtain
\begin{align}\label{eq:decayupper}
\begin{split}
 |\widehat\nu(\xi)|
&\le |M_{B_{n-1}}(t)M_{B_n}(t/N_n)M_{B_{n+1}}(t/(N_n N_{n+1}))|\\
&\lesssim N_n^{-1}+\eta_n+\eta_{n+1}
 \lesssim P_n^{-\varepsilon}+P_{n+1}^{-\varepsilon}\lesssim\xi^{-\varepsilon}.    
\end{split}
\end{align}
For the remaining bounded range $0\le\xi<P_{n_0-1}$, we have
\begin{align*}
 |\widehat\nu(\xi)|\le
 1\le (1+P_{n_0-1})^{\varepsilon}    (1+\xi)^{-\varepsilon}.
\end{align*}
The estimate for negative frequencies follows from \(\widehat{\nu}(-\xi)=\overline{\widehat{\nu}(\xi)}\). Finally, let \(T(x)=1+x/4\) and define \(\sigma=T_\#\nu\), where \(T_\#\nu\) denotes the pushforward of \(\nu\) under \(T\). Then
\[
\widehat{\sigma}(\xi)=e^{-2\pi i\xi}\widehat{\nu}(\xi/4).
\]
We have therefore obtained all the claimed one-dimensional Fourier estimates for this one fixed measure.

\par\medskip\noindent\textit{Step 3. Cylinders, gaps, and singularity.} We use the growing separation between digits to construct disjoint cylinders covering the support of $\sigma$.
Let $r_n=(4P_n)^{-1}$. For each word $w=(b_1,\ldots,b_n)$, define the level-$n$ cylinder
\begin{equation}\label{eq:cylinder}
 I_{n,w}:=[a_{n,w},a_{n,w}+r_n],\qquad \text{where  }
 a_{n,w}:=1+\frac14\sum_{j=1}^n\frac{b_j}{P_j}.
\end{equation}
By \eqref{eq:tail}, every digit sequence determines a point in one of these cylinders. Since their union is a finite union of closed intervals, it contains \(\operatorname{supp}\sigma\). 

We next show that distinct level-$n$ cylinders are well separated. Let \(w=(b_1,\ldots,b_n)\) and \(w'=(b'_1,\ldots,b'_n)\) be two distinct words, and let \(i\) be the first index at which they differ. Thus \(b_j=b'_j\) for \(j<i\), while \(b_i\ne b'_i\).
By the digit separation property in Lemma~\ref{lem:digits}, we have
\[
|b_i-b'_i|\ge h_i/2.
\]
On the other hand, the contribution from all later digits satisfies
\[
\sum_{j=i+1}^n\frac{|b_j-b'_j|}{P_j}
\le \sum_{j>i}\frac{N_j-1}{P_j}
=\frac1{P_i},
\]
where the last equality follows from \eqref{eq:tail}.
Therefore, the difference at the first distinct digit cannot be canceled by the later digits. More precisely,
\begin{equation}\label{new2.14}
 4|a_{n,w}-a_{n,w'}|
=\left|\sum_{j=i}^n\frac{b_j-b'_j}{P_j}\right|\ge \frac{|b_i-b_i'|}{P_i}-\sum_{j=i+1}^n\frac{|b_j-b'_j|}{P_j}
\ge\frac{h_i/2-1}{P_i}.   
\end{equation}
If \(i<n\), then \(P_n/P_i\ge N_n\ge h_n\). Since \(h_i\ge64\), it follows that
\[
4P_n|a_{n,w}-a_{n,w'}|
\ge N_n(h_i/2-1)
\ge h_n/2-1.
\]
If \(i=n\), the same lower bound follows directly from \eqref{new2.14}.
Recalling that \(r_n=(4P_n)^{-1}\), we conclude that
\begin{equation}\label{eq:gaps}
|a_{n,w}-a_{n,w'}|\ge g_nr_n,
\qquad
g_n:=h_n/2-1\longrightarrow\infty.
\end{equation}

Thus the level-\(n\) cylinders are pairwise separated by a distance of at least \((g_n-1)r_n\), and hence are pairwise disjoint.
Let
$Q_n:=\prod_{j=1}^n m_j$
be the number of level-\(n\) cylinders. Since each cylinder corresponds to a unique prescribed prefix, independence and uniform distribution give 
\begin{equation}\label{eq:mass}
 \sigma(I_{n,w})=Q_n^{-1}.
\end{equation}
Finally, since the level-\(n\) cylinders cover \(\operatorname{supp}\sigma\) and each has length \(r_n\), we obtain
\begin{equation}\label{eq:singular}
 |\supp\sigma|\le Q_nr_n=\frac14\prod_{j=1}^n h_j^{-1}\xrightarrow[n\to\infty]{} 0.
\end{equation}
Hence \(\sigma\) is singular with respect to Lebesgue measure.
Moreover, since \(P_n\ge N_n\ge h_n\), we have
\[
r_n\sqrt{g_n}
=\frac{\sqrt{h_n/2-1}}{4P_n}
\le\frac1{4\sqrt{2h_n}}
\xrightarrow[n\to\infty]{} 0.
\]
\end{proof}

\subsection{\texorpdfstring{The $L^1$ weak endpoint}{The L1 weak endpoint}}\label{sub:l1proof}
Fix the one-dimensional $\sigma$ from \Lemref{lem:fixedmeasure}. We first establish the exact constant for arbitrary finite dyadic families.
\begin{lemma}[Exact finite-scale $L^1$ constant]\label{lem:l1sharp}
For every integer $d\ge1$, let $\mu_d:=\sigma^{\otimes d}$. Then $\mathfrak C_S(\mu_d)=\#S$ for every finite nonempty $S\subseteq\Z$.
\end{lemma}
\begin{proof}[Proof of \Lemref{lem:l1sharp}]
\leavevmode\par\smallskip\noindent \par\medskip\noindent\textit{Step 1. Product measures and the common scale shift.} We take the $d$-fold product of $\sigma$:
\begin{equation}\label{eq:productmeasure}
 \mu_d:=\sigma^{\otimes d},\qquad
 \widehat{\mu_d}(\xi)=\prod_{\ell=1}^d\widehat\sigma(\xi_\ell).
\end{equation}
At least one coordinate satisfies $|\xi_\ell|\ge|\xi|/\sqrt d$. Applying the 1D bound to that coordinate and $|\widehat\sigma|\le1$ to the others proves \Eqref{eq:allpowers}. The support is contained in $[1,5/4]^d$ and is Lebesgue-null by Fubini. A product cylinder has side length $r_n$ and mass $Q_n^{-d}$.

We first remove the position of the finite scale family. A common integer shift does not change either weak constant. To see this, put $F(x):=f(2^a x)$ and compute
\begin{equation}\label{eq:scaleshift}
 (f*(\mu_d)_{k+a})(2^a x)=(F*(\mu_d)_k)(x),
 \qquad\norm F_1=2^{-ad}\norm f_1.
\end{equation}
The same change of variables works for Hardy atoms. Indeed, $x\mapsto2^{ad}b(2^a x)$ is again an atom whenever $b$ is. Applying this observation to atomic decompositions, we get
$\norm{f(2^a\cdot)}_{\Ha}\le2^{-ad}\norm f_{\Ha}$. Applying the same argument with $-a$ yields the reverse inequality. Thus both weak constants are unchanged by the common index shift. We may assume $S\subseteq\{0,\ldots,K\}$ and write $J=\#S$.

We now fix $S$ and choose the stage $n$ sufficiently large (depending on $S$). Write $g:=g_n$, $r:=r_n$, and $Q:=Q_n$. Define
\begin{equation}\label{eq:tests}
 \delta:=2^Kr\sqrt g,\qquad q:=1-g^{-1/2},\qquad
 f:=\ind{[0,\delta]^d},\qquad \lambda:=qQ^{-d}.
\end{equation}
\par\medskip\noindent\textit{Step 2. One positive output box per cylinder and scale.} For each scale $k\in S$ and product cylinder $w=(w_1,\ldots,w_d)$, we place the output box
\begin{equation}\label{eq:goodboxes}
 G_{k,w}:=\prod_{\ell=1}^d
 [2^k(a_{n,w_\ell}+r),\,2^k a_{n,w_\ell}+\delta].
\end{equation}
If $x\in G_{k,w}$, the translated input $x-2^ky$ belongs to $[0,\delta]^d$ for every $y$ in the generating product cylinder. Hence the entire cylinder mass contributes, and we have
\begin{equation}\label{eq:boxproperties}
 (f*(\mu_d)_k)(x)\ge Q^{-d}>\lambda\quad(x\in G_{k,w}),
 \qquad |G_{k,w}|=(\delta-2^kr)^d\ge q^d\delta^d.
\end{equation}
To add the contributions of different cylinders, we must prove that the output boxes do not overlap. For distinct words at the same scale, at least one coordinate is separated by $(g-1)r$. The corresponding output intervals are disjoint because
\begin{equation}\label{eq:boxseparation}
 \frac{\delta}{2^k(g-1)r}
 \le\frac{2^K}{\sqrt g}\frac g{g-1}<1.
\end{equation}
At different scales, we use the location of the first coordinate. Its output intervals lie in
$[2^k,(5/4)2^k+\delta]$, which are disjoint when $\delta<1/4$. Both requirements hold eventually, since $K$ is fixed. Therefore the $JQ^d$ boxes are pairwise disjoint. Since $\norm f_1=\delta^d$, summing their volumes gives
\begin{equation}\label{eq:weaklower}
 \frac{\lambda\bigl|\{x:\max_{k\in S}|f*(\mu_d)_k(x)|>\lambda\}\bigr|}{\norm f_1}
 \ge\frac{qQ^{-d}}{\delta^d}\,JQ^d q^d\delta^d
 =Jq^{d+1}\longrightarrow J.
\end{equation}
\par\medskip\noindent\textit{Step 3. Matching upper bound.} For any probability measure $\tau$, Tonelli's theorem and the union bound give
\begin{equation}\label{eq:union}
 \lambda\bigl|\{x:\max_{k\in S}|f*\tau_k(x)|>\lambda\}\bigr|
 \le\sum_{k\in S}\norm{|f|*\tau_k}_1=J\norm f_1.
\end{equation}
Combining the lower bound with the general upper bound, and then taking $n\to\infty$ while $S$ stays fixed, we conclude $\mathfrak C_S(\mu_d)=J$.

\end{proof}

\subsection{The Hardy endpoint and conclusion}\label{sub:hardyproof}
The Hardy endpoint needs a genuinely cancellative test. The next argument proves the exact constant using one cube atom for each finite family.
\begin{lemma}[Exact finite-scale atomic Hardy constant]\label{lem:hardysharp}
For the same fixed $\mu_d$ and every finite nonempty $S\subseteq\Z$, we have $\mathfrak H_S(\mu_d)=\#S$.
\end{lemma}
\begin{proof}[Proof of \Lemref{lem:hardysharp}]
\leavevmode\par\smallskip\noindent \par\medskip\noindent\textit{Step 1. A normalized cancellative test.} The previous positive test is not a Hardy atom. We therefore build a mean-zero cube atom and use both signs of its convolution. Fix a nonempty finite $S\subseteq\{0,\ldots,K\}$, write $J:=\#S$ and use $Q=Q_n$, $r=r_n$, $g=g_n$, $\delta=2^Kr\sqrt g$ from \Eqref{eq:tests}. Define
\begin{equation}\label{eq:hardyatom}
 a_\delta(x):=(2\delta)^{-d}
 \bigl(\ind{[0,\delta]}(x_1)-\ind{(\delta,2\delta]}(x_1)\bigr)
 \prod_{\ell=2}^d\ind{[0,2\delta]}(x_\ell).
\end{equation}
For $d=1$, the empty product equals one. Direct computation gives
\begin{equation}\label{eq:atomnormalization}
 \supp a_\delta\subseteq[0,2\delta]^d,\qquad
 \int a_\delta=0,\qquad
 \norm{a_\delta}_\infty=(2\delta)^{-d},\qquad
 \norm{a_\delta}_1=\norm{a_\delta}_{\Ha}=1.
\end{equation}
The first three identities show that $a_\delta$ is a normalized Hardy atom. The equality $\norm{a_\delta}_{\Ha}=1$ then follows from $\norm{a_\delta}_1=1$ and \Eqref{eq:atomicembed}.

\par\medskip\noindent\textit{Step 2. Positive and negative boxes.} For a product cylinder with lower corner $(a_{w_1},\ldots,a_{w_d})$, we choose one box for each sign of the atom:
\begin{align}\label{eq:hardyboxes}
 G_{k,w}^{+}&:=[2^k(a_{w_1}+r),\,2^ka_{w_1}+\delta]
 \times\prod_{\ell=2}^d[2^k(a_{w_\ell}+r),\,2^ka_{w_\ell}+2\delta],\\
 G_{k,w}^{-}&:=[2^k(a_{w_1}+r)+\delta,\,2^ka_{w_1}+2\delta]
 \times\prod_{\ell=2}^d[2^k(a_{w_\ell}+r),\,2^ka_{w_\ell}+2\delta].\nonumber
\end{align}
The boundaries have measure zero and will play no role. Both boxes have volume
\begin{equation}\label{eq:hardyboxvolume}
 |G_{k,w}^{+}|=|G_{k,w}^{-}|=(\delta-2^kr)(2\delta-2^kr)^{d-1}.
\end{equation}
We now choose $n$ large enough that the following three inequalities hold simultaneously:
\begin{equation}\label{eq:hardyseparation}
 2^k(g-1)r>2\delta\quad(k\in S),\qquad
 2\delta<\tfrac14,\qquad
 2^Kr<\delta.
\end{equation}
Indeed, $2\delta/[2^k(g-1)r]\le 2^{K+1}\sqrt g/(g-1)\to0$ and $\delta\to0$ for fixed $K$.

If $x\in G_{k,w}^{+}$, then for every $y$ in the generating product cylinder $I_{n,w}$,
\begin{equation}\label{eq:hardypositive}
 x_1-2^ky_1\in[0,\delta],\qquad
 x_\ell-2^ky_\ell\in[0,2\delta]\quad(2\le\ell\le d).
\end{equation}
For $x$ in the interior of $G_{k,w}^{-}$, the first coordinate instead lies in $(\delta,2\delta)$ after translation. \par\medskip\noindent\textit{Step 3. Exact exclusion of other cylinders.} We must also rule out cancellation from \emph{other} cylinders. Choose one coordinate $\ell$ in which $v_\ell\ne w_\ell$. For either output box, we have $2^ka_{w_\ell}\le x_\ell\le2^ka_{w_\ell}+2\delta$. Using the gap estimate from \Lemref{lem:fixedmeasure}, the two possible orders give
\begin{equation}\label{eq:hardyexclusion}
 \begin{aligned}
 a_{v_\ell}\ge a_{w_\ell}+gr
 &\ \Longrightarrow\ x_\ell-2^ky_\ell\le2\delta-2^kgr<0,\\
 a_{v_\ell}\le a_{w_\ell}-gr
 &\ \Longrightarrow\ x_\ell-2^ky_\ell\ge2^k(g-1)r>2\delta.
 \end{aligned}
\end{equation}
In either case the translated point lies outside the atom's supporting cube. Thus every unwanted cylinder contributes \emph{exactly zero}, while the selected one contributes with the required sign. We obtain the exact identities
\begin{equation}\label{eq:hardysigned}
 (a_\delta*(\mu_d)_k)(x)=
 \begin{cases}
 +(2\delta)^{-d}Q^{-d},&x\in G_{k,w}^{+},\\
 -(2\delta)^{-d}Q^{-d},&x\in G_{k,w}^{-}.
 \end{cases}
\end{equation}
We next check that no two displayed boxes overlap. At a fixed scale, the positive and negative boxes are disjoint. Distinct words are separated in at least one coordinate by $2^k(g-1)r>2\delta$. Across scales, the first-coordinate slabs lie in $[2^k,(5/4)2^k+2\delta]$. When $k'>k\ge0$, we have
$(5/4)2^k+2\delta<(3/2)2^k<2^{k+1}\le2^{k'}$ by \Eqref{eq:hardyseparation}. We have therefore obtained $2JQ^d$ pairwise disjoint signed output boxes.

\par\medskip\noindent\textit{Step 4. The sharp lower bound.} Set $q:=1-g^{-1/2}$ and take the strict threshold $\lambda=q(2\delta)^{-d}Q^{-d}$. Because $2^kr/\delta\le g^{-1/2}$ for every $k\in S$, we can sum the exact box volumes from \Eqref{eq:hardyboxvolume} and obtain
\begin{align}\label{eq:hardysharp}
 &\lambda\bigl|\{x:\max_{k\in S}|a_\delta*(\mu_d)_k(x)|>\lambda\}\bigr|\\
 &\quad\ge q(2\delta)^{-d}Q^{-d}\sum_{k\in S}\sum_w
 2(\delta-2^kr)(2\delta-2^kr)^{d-1}\\
 &\quad\ge J(1-g^{-1/2})^2
 \bigl(1-\tfrac12g^{-1/2}\bigr)^{d-1}\longrightarrow J.\nonumber
\end{align}
Because $\norm{a_\delta}_{\Ha}=1$, we obtain $\mathfrak H_S(\mu_d)\ge J$. Conversely,
\begin{equation}\label{eq:hardyupper}
 \lambda|\{\max_{k\in S}|f*(\mu_d)_k|>\lambda\}|
 \le J\norm f_1\le J\norm f_{\Ha}
\end{equation}
by the general union bound. Hence $\mathfrak H_S(\mu_d)=J$. We also obtain the same identity for arbitrary finite $S\subseteq\Z$ by the common dilation in \Eqref{eq:scaleshift}.

\end{proof}

\begin{proof}[Proof of \Thmref{thm:main}]
Take $\sigma$ from \Lemref{lem:fixedmeasure} and set $\mu_d:=\sigma^{\otimes d}$. The product argument at the beginning of \Lemref{lem:l1sharp} proves \Eqref{eq:allpowers}. The construction gives a compact Lebesgue-null support inside $[1,5/4]^d$. Furthermore, \Lemref{lem:l1sharp} and \Lemref{lem:hardysharp} give \Eqref{eq:exact} for every finite nonempty $S$. By taking finite subsets with $\#S\to\infty$, we see that neither full dyadic maximal operator has a finite weak endpoint norm.

Finally, the one-dimensional support has the asserted full dimension. For $0<s<1$, 
the $s$-Riesz energy of $\sigma$ is defined by
\[
I_s(\sigma)
:=
\iint_{\mathbb R^2}
\frac{d\sigma(x)\,d\sigma(y)}{|x-y|^s}.
\]
Choose $s/2<\varepsilon<1/2$. The Fourier formula for Riesz energy \cite{Mattila} gives
\begin{equation}\label{eq:riesz}
 I_s(\sigma)=c_s\int_\R|\widehat\sigma(\xi)|^2|\xi|^{s-1}\dd\xi
 \lesssim_s1+C_\varepsilon^2\int_1^\infty t^{s-1-2\varepsilon}\dd t<\infty.
\end{equation}
Hence $\dim_H(\supp\sigma)\ge s$. The all-power estimate also gives Fourier dimension at least $2\varepsilon$ for every $\varepsilon<1/2$. Since Fourier dimension does not exceed Hausdorff dimension,
\begin{equation}\label{eq:dimensions}
 \dim_F(\supp\sigma)=\dim_H(\supp\sigma)=1.
\end{equation}
This completes the proof of all parts of \Thmref{thm:main}.

\end{proof}

\section{Proofs of the two Fourier lemmas}\label{sec:auxiliary}
\subsection{Selecting separated renewal digits}\label{sec:renewal}
\begin{proof}[Proof of \Lemref{lem:digits}]
We fix $h$ and construct the digits using a random walk with positive integer steps. We first prove stability and bias bounds for the step characteristic function. Then we obtain concentration uniformly over the entire frequency circle and select a sample with the required endpoint.
\par\medskip\noindent\textit{Step 1. Uniform characteristic-function stability.}
Fix $h\ge64$ divisible by eight. Let $Z_1,\ldots,Z_m$ be independent, each uniform on the $h+1$ integers in $[h/2,3h/2]$. Set
\begin{equation}\label{eq:renewal}
 X_0:=0,\quad X_j:=\sum_{i=1}^jZ_i,\quad N:=mh,\quad
 S_m(t):=\frac1m\sum_{j=1}^m e^{-2\pi iX_jt}.
\end{equation}
The characteristic function of a step is
\begin{equation}\label{eq:phidefinition}
 \phi(t):=\mathbb E e^{-2\pi iZ_1t}=e^{-2\pi iht}a_h(t),\qquad
 a_h(t):=\frac1{h+1}\sum_{v=-h/2}^{h/2}e^{-2\pi ivt}\in\R.
\end{equation}
We claim the absolute, uniform estimate
\begin{equation}\label{eq:stability}
 |e^{-2\pi izt}-\phi(t)|\le C|1-\phi(t)|
 \quad(z\in[h/2,3h/2]\cap\Z,\ t\in\R).
\end{equation}
Both sides are periodic, and magnitudes are unchanged by conjugation. It suffices to take $0<t\le1/2$. For $v=ht\le1/100$, symmetry gives
\begin{equation}\label{eq:smallstability}
 \begin{split}
 1-a_h(t)&=O(v^2),\qquad a_h(t)\ge1/2,\\
 |1-\phi(t)|&\ge a_h(t)|\sin(2\pi v)|\gtrsim v,\qquad
 |e^{-2\pi izt}-\phi(t)|\lesssim v.
 \end{split}
\end{equation}
For $1/100\le ht\le1$, use
\begin{equation}\label{eq:varianceidentity}
 1-|a_h(t)|^2
 =\frac1{(h+1)^2}\sum_{p,q=-h/2}^{h/2}
 \bigl(1-\cos(2\pi(p-q)t)\bigr).
\end{equation}
Restrict $p$ to $[-h/4,-h/8]\cap\Z$ and $q$ to $[0,h/8]\cap\Z$. These pairs have a fixed positive proportion, and
\[
 h/8\le|p-q|\le3h/8,\qquad
 1-\cos(2\pi(p-q)t)\gtrsim(ht)^2.
\]
Thus $1-|a_h(t)|^2\gtrsim10^{-4}$, so $|1-\phi(t)|\ge1-|a_h(t)|\gtrsim1$. The implicit lower bound is absolute, not dependent on $h$.

For $1\le ht\le h/2$, the geometric-sum formula gives
\begin{equation}\label{eq:largestability}
 |a_h(t)|=\frac{|\sin(\pi(h+1)t)|}{(h+1)|\sin(\pi t)|}
 \le\frac1{2(h+1)t}\le\frac12.
\end{equation}
Hence $|1-\phi(t)|\ge1/2$. The numerator in \Eqref{eq:stability} never exceeds two. At integer $t$ both sides are zero. This proves \Eqref{eq:stability} everywhere.

\par\medskip\noindent\textit{Step 2. Control of the expectation.}
Independence gives
\begin{equation}\label{eq:renewalmean}
 \mathbb ES_m(t)=\frac1m\sum_{j=1}^m\phi(t)^j.
\end{equation}
We prove
\begin{equation}\label{eq:bias}
 \sup_{t\in\R}|\mathbb ES_m(t)-\DM N(t)|\lesssim m^{-1/2}.
\end{equation}
Reduce again to $t\in[0,1/2]$. If $t\ge1/(100h)$, the denominator estimates just proved yield
\begin{equation}\label{eq:biaslarge}
 |\mathbb ES_m(t)|\le\frac2{m|1-\phi(t)|}\lesssim m^{-1},\qquad
 |\DM N(t)|\lesssim(Nt)^{-1}\lesssim m^{-1}.
\end{equation}
If $0<t<1/(100h)$, write $v=ht$ and $q=e^{-2\pi iv}$. Then
\begin{equation}\label{eq:biasparameters}
 |\phi-q|\lesssim v^2,\qquad |1-\phi|\eqsim|1-q|\eqsim v,
 \qquad\DM N(t)=\DM m(ht)\DM h(t).
\end{equation}
For $A(z)=z/(1-z)$, the exact difference of geometric sums is
\begin{align}\label{eq:geometricdifference}
 \frac1m\sum_{j=1}^m(\phi^j-q^j)
 &=\frac1m\bigl(A(\phi)-A(q)\bigr)(1-\phi^m)
   +\frac1m A(q)(q^m-\phi^m),\\
 |A(\phi)-A(q)|&=\frac{|\phi-q|}{|1-\phi|\,|1-q|}\lesssim1,\qquad
 |q^m-\phi^m|\le\min\{m|\phi-q|,2\}.\nonumber
\end{align}
Therefore
\begin{equation}\label{eq:biasfirst}
 \left|\frac1m\sum_{j=1}^m(\phi^j-q^j)\right|
 \lesssim m^{-1}+\min\{v,(mv)^{-1}\}\lesssim m^{-1/2}.
\end{equation}
Also, $|q-\DM h(t)|\lesssim v$ and
\begin{equation}\label{eq:biassecond}
 \left|\frac1m\sum_{j=1}^m q^j-\DM N(t)\right|
 =|\DM m(ht)|\,|q-\DM h(t)|
 \lesssim\min\{1,(mv)^{-1}\}v\lesssim m^{-1}.
\end{equation}
At integers both masks are one. Equations \Eqref{eq:biaslarge}--\Eqref{eq:biassecond} prove \Eqref{eq:bias}.

\par\medskip\noindent\textit{Step 3. Martingale differences and the frequency grid.}
Let $W_m(t)=mS_m(t)$ and $\mathcal F_i=\sigma(Z_1,\ldots,Z_i)$. Conditional expectation gives
\begin{equation}\label{eq:conditionalexpectation}
 \mathbb E(W_m(t)\mid\mathcal F_i)
 =\sum_{j=1}^i e^{-2\pi iX_jt}
  +e^{-2\pi iX_it}\sum_{r=1}^{m-i}\phi(t)^r.
\end{equation}
Subtracting the formula at $i-1$ produces exactly
\begin{equation}\label{eq:doobexact}
 \begin{split}
 &\mathbb E(W_m(t)\mid\mathcal F_i)-\mathbb E(W_m(t)\mid\mathcal F_{i-1})\\
 &\qquad=e^{-2\pi iX_{i-1}t}
 (e^{-2\pi iZ_it}-\phi(t))\sum_{r=0}^{m-i}\phi(t)^r.
 \end{split}
\end{equation}
For noninteger $t$, the geometric sum is at most $2/|1-\phi(t)|$, and \Eqref{eq:stability} bounds every increment by an absolute constant. At integers the increments vanish. Applying the bounded-increment inequality of Azuma \cite{Azuma} to real and imaginary parts yields
\begin{equation}\label{eq:azuma}
 \mathbb P\bigl(|S_m(t)-\mathbb ES_m(t)|>u\bigr)\le4e^{-cmu^2},
 \qquad t\in\R.
\end{equation}
Every realization has $0<X_j\le3N/2$. In particular,
$\norm{S_m'}_\infty\le3\pi N$, $\norm{(\mathbb ES_m)'}_\infty\le3\pi N$ and $\norm{(\DM N)'}_\infty\le2\pi N$. On $\mathcal G_N=\{j/N^3:0\le j\le N^3\}$, choose $u=K\sqrt{\log N/m}$. Equation \Eqref{eq:azuma} yields
\begin{equation}\label{eq:gridunion}
 \mathbb P\left(\max_{t\in\mathcal G_N}|S_m(t)-\mathbb ES_m(t)|>u\right)
 \le4(N^3+1)N^{-cK^2}\le N^{-2}
\end{equation}
for all sufficiently large $N$, after fixing the absolute $K$ so that $cK^2\ge7$. The bias in \Eqref{eq:bias} and the derivative bound add at most $Cm^{-1/2}+CN^{-2}$ to the supremum. They are absorbed by an absolute multiple of $\sqrt{\log N/m}$ because $m\le N$. Thus
\begin{equation}\label{eq:uniformprobability}
 \mathbb P\left(\sup_{t\in\R}|S_m(t)-\DM N(t)|
   >C_1\sqrt{\frac{\log N}{m}}\right)\le N^{-2}
\end{equation}
for a sufficiently large absolute $C_1$ and sufficiently large $N$. Here $m=N/h\le N$, so the grid error is absorbed by the displayed threshold. Periodicity extends the estimate from $[0,1]$ to every real frequency.

\par\medskip\noindent\textit{Step 4. Selection inside the base interval.}
The random sum $X_m-N=\sum_{i=1}^m(Z_i-h)$ is symmetric. For independent integer-valued variables, the maximum atom of a convolution does not exceed the maximum atom of either factor. Hence
\begin{equation}\label{eq:insideevent}
 \mathbb P(X_m=N)\le(h+1)^{-1},\qquad
 \mathbb P(X_m<N)=\tfrac12\bigl(1-\mathbb P(X_m=N)\bigr)\ge\tfrac13.
\end{equation}
The event in \Eqref{eq:insideevent} intersects the good event in \Eqref{eq:uniformprobability} for large $N$. Select one realization in the intersection and put $B=\{X_1,\ldots,X_m\}$. Then
\begin{equation}\label{eq:digitfinish}
 0<X_1<\cdots<X_m<N,\qquad X_{j+1}-X_j=Z_{j+1}\ge h/2,
 \qquad\BM B=S_m.
\end{equation}
The selected points are distinct, lie in the required integer interval, and have gaps at least $h/2$. Together with \Eqref{eq:uniformprobability}, this proves \Lemref{lem:digits}. Notice that concentration was obtained \emph{before} selecting the event $X_m<N$; we did not condition the martingale on that event.

\end{proof}

\subsection{The adjacent-mask estimates}\label{sec:adjacent}
\begin{proof}[Proof of \Lemref{lem:adjacent}]
We first compare two neighboring masks in the lower half of a block. In the upper half, we reflect toward the nearest integer endpoint and use a third factor.
For $A,B\ge2$ and $1\le t\le B/2$, put $d(t)=\dist(t,\Z)$. The Dirichlet identities imply
\begin{equation}\label{eq:Dirichletbounds}
 |\DM A(t)|\le\min\{1,(2A d(t))^{-1}\},\qquad
 |\DM B(t/B)|\le C\frac{d(t)}t.
\end{equation}
For the second estimate, use
\[
 |\DM B(t/B)|=\frac{|\sin(\pi t)|}{B|\sin(\pi t/B)|},\quad
 |\sin(\pi t)|\le\pi d(t),\quad B\sin(\pi t/B)\ge2t.
\]
At integral $t\in[1,B/2]$, the numerator is zero and the denominator is nonzero.

Write $\BM{B_j}=\DM{N_j}+e_j$, with $\norm{e_j}_\infty\le\eta_j$. Using $|\BM{B_{n-1}}|\le1$ for the error in the second factor gives
\begin{align}\label{eq:pairproof}
 |\BM{B_{n-1}}(t)\BM{B_n}(t/N_n)|
 &\le |\DM{N_{n-1}}(t)\DM{N_n}(t/N_n)|
       +\eta_{n-1}|\DM{N_n}(t/N_n)|+\eta_n\\
 &\lesssim\frac1{N_{n-1}t}+\frac{\eta_{n-1}}t+\eta_n.\nonumber
\end{align}
We have proved \Eqref{eq:pair} also when $t$ is an integer, because \Eqref{eq:Dirichletbounds} was interpreted at those points without division by $d(t)$.

For $N_n/2<t<N_n$, set $u=N_n-t$ and $b=t/N_n=1-u/N_n$. Integer digits give the exact reflection identity
\begin{equation}\label{eq:reflection}
 |\BM{B_{n-1}}(t)\BM{B_n}(t/N_n)|
 =|\BM{B_{n-1}}(u)\BM{B_n}(u/N_n)|.
\end{equation}
Since $1/2<b<1$,
\begin{equation}\label{eq:thirdmask}
 |\DM{N_{n+1}}(b/N_{n+1})|
 =\frac{|\sin(\pi b)|}{N_{n+1}|\sin(\pi b/N_{n+1})|}
 \le C(1-b)=C u/N_n.
\end{equation}
This gives the third-mask upper bound $Cu/N_n+\eta_{n+1}$. The factor matters precisely when $u=N_n-t$ approaches zero.

If $0<u<1$, that factor alone bounds the triple product by $C/N_n+\eta_{n+1}$. If $1\le u<N_n/2$, let $A_n=N_{n-1}^{-1}+\eta_{n-1}\le3/2$. Applying \Eqref{eq:pair} at $u$ gives
\begin{align}\label{eq:tripleproof}
  |\BM{B_{n-1}}(t)\BM{B_n}(t/N_n)\BM{B_{n+1}}(b/N_{n+1})|
 &\lesssim(A_n/u+\eta_n)(u/N_n+\eta_{n+1})\\
 &= A_n/N_n+A_n\eta_{n+1}/u
     +\eta_nu/N_n+\eta_n\eta_{n+1}\\
 &\lesssim N_n^{-1}+\eta_n+\eta_{n+1}.\nonumber
\end{align}
In the last step we used only $u\ge1$, $u\le N_n/2$, and $\eta_j\le1$. The separate case $0<u<1$ covers the immediate resonant edge, completing \Eqref{eq:triple}.

\end{proof}

\medskip
\noindent\textbf{Conflicts of Interest}\space The authors declare that there are no conflicts of interest.

\medskip
\noindent{\bf AI disclosure.}
During the preparation of this manuscript, the authors used ChatGPT and Codex to discuss possible proof strategies, assist in drafting and checking parts of the mathematical arguments, organize references, and improve the clarity and readability of the exposition. All mathematical statements and proofs in the final manuscript were written, reviewed, and verified by the authors, who take full responsibility for the content of this manuscript.


\medskip
\begin{thebibliography}{99}\raggedright

\bibitem{Azuma} K. Azuma, Weighted sums of certain dependent random variables, Tohoku Math. J. (2) 19 (3) (1967) 357--367. \doi{10.2748/tmj/1178243286}.

\bibitem{CLLS} S. S. Choudhary, J. Li, C.-W. Liang and C.-Y. Shen, Lacunary $\delta$-discretised spherical maximal operators, preprint (2025), revised September 13, 2026. \arxiv{2507.09962v2}.

\bibitem{Christ} M. Christ, Weak type $(1,1)$ bounds for rough operators, Ann. of Math. (2) 128 (1) (1988) 19--42. \doi{10.2307/1971461}.

\bibitem{CK} L. Cladek and B. Krause, Improved endpoint bounds for the lacunary spherical maximal operator, Anal. PDE 17 (2024) 2011--2032. \doi{10.2140/apde.2024.17.2011}.

\bibitem{DR} J. Duoandikoetxea and J. L. Rubio de Francia, Maximal and singular integral operators via Fourier transform estimates, Invent. Math. 84 (3) (1986) 541--561. \doi{10.1007/BF01388746}.

\bibitem{GOPSS} L. Grafakos, D. Oliveira e Silva, M. Pramanik, A. Seeger and B. Stovall, Some problems in harmonic analysis, preprint (2017). \arxiv{1701.06637v1}.

\bibitem{HKLL} S. Ham, J. Kah, S. Lee and J. Li, On Rubio de Francia's maximal theorem, preprint (2026). \arxiv{2602.03465v1}.

\bibitem{HL} K. Hambrook and I. \L aba, On the sharpness of Mockenhaupt's restriction theorem, Geom. Funct. Anal. 23 (4) (2013) 1262--1277. \doi{10.1007/s00039-013-0240-9}.

\bibitem{LParith} I. \L aba and M. Pramanik, Arithmetic progressions in sets of fractional dimension, Geom. Funct. Anal. 19 (2) (2009) 429--456. \doi{10.1007/s00039-009-0003-9}.

\bibitem{LPmax} I. \L aba and M. Pramanik, Maximal operators and differentiation theorems for sparse sets, Duke Math. J. 158 (3) (2011) 347--411. \doi{10.1215/00127094-1345644}.

\bibitem{LSX} C.-K. Lai, R. Shi and Y.-H. Xie, A new type of deterministic Salem sets and its spectrality, preprint (2026). \arxiv{2609.22772v1}.

\bibitem{LLLS} S. Lee, J. Li, C.-W. Liang and C.-Y. Shen, On the endpoint estimate for discrete spherical average over sparse sequences, preprint (2026), revised October 2, 2026. \arxiv{2608.03004v5}.

\bibitem{Mattila} P. Mattila, Fourier Analysis and Hausdorff Dimension, Cambridge Studies in Advanced Mathematics, vol. 150, Cambridge University Press, Cambridge, 2015. \doi{10.1017/CBO9781316227619}.

\bibitem{STW} A. Seeger, T. Tao and J. Wright, Pointwise convergence of lacunary spherical means, in: Harmonic Analysis at Mount Holyoke, Contemporary Mathematics, vol. 320, American Mathematical Society, Providence, RI, 2003, pp. 341--351. \doi{10.1090/conm/320/05617}.

\bibitem{SW} A. Seeger and J. Wright, Problems on averages and lacunary maximal functions, Banach Center Publ. 95 (2011) 235--250. \doi{10.4064/bc95-0-11}.

\bibitem{Stein} E. M. Stein, Harmonic Analysis: Real-Variable Methods, Orthogonality, and Oscillatory Integrals, Princeton Mathematical Series, vol. 43, Princeton University Press, Princeton, NJ, 1993. \mbox{\href{https://press.princeton.edu/books/paperback/9780691113876/harmonic-analysis}{Publisher}}.

\end{thebibliography}
\end{document}